\documentclass[a4paper,12pt]{amsart}

\usepackage{amssymb}
\usepackage{amsmath}
\usepackage[utf8]{inputenc}
\usepackage[colorlinks=true]{hyperref}

\usepackage[
 margin=1in,
 includefoot,
 footskip=15pt,
]{geometry}
\usepackage{xcolor}
\usepackage{enumerate}
\usepackage{mathtools}

\makeatletter
\renewcommand*{\eqref}[1]{%
	\hyperref[{#1}]{\textup{\tagform@{\ref*{#1}}}}%
}
\makeatother

\newtheorem{theorem}{Theorem}[section]
\newtheorem{lemma}[theorem]{Lemma}
\newtheorem{proposition}[theorem]{Proposition}
\newtheorem{corollary}[theorem]{Corollary}

\theoremstyle{definition}
\newtheorem{definition}[theorem]{Definition}
\newtheorem{example}[theorem]{Example}

\theoremstyle{remark}
\newtheorem{remark}[theorem]{Remark}

\numberwithin{equation}{section}

\newcommand{\IFS}{(\gamma_1,\ldots,\gamma_d)}
\newcommand{\hpot}{(h_1,\ldots,h_d)}

\newcommand{\rn}{\mathbb{R}}
\newcommand{\cn}{\mathbb{C}}
\newcommand{\nn}{\mathbb{N}}

\newcommand{\scj}{\subseteq}

\newcommand{\CF}{\mathcal{F}}
\newcommand{\CG}{\mathcal{G}}

\newcommand{\CL}{\mathcal{L}}
\newcommand{\CO}{\mathcal{O}}
\newcommand{\CS}{\mathcal{S}}
\newcommand{\CT}{\mathcal{T}}

\newcommand{\Sp}{\operatorname{Sp}}
\newcommand{\Tr}{\operatorname{Tr}}

\makeatletter
\@namedef{subjclassname@2020}{%
	\textup{2020} Mathematics Subject Classification}
\makeatother

\begin{document}

\title[KMS states on C*-algebras of self-similar sets]{KMS states for generalized gauge actions on C*-algebras associated with self-similar sets}

\author{Gilles G. de Castro}
\address{Departamento de Matem\'atica, Universidade Federal de Santa Catarina, 88040-970 Florian\'opolis SC, Brazil.}
\email{gilles.castro@ufsc.br}
\thanks{Partially supported by CAPES}

\begin{abstract}
Given a self-similar $K$ set defined from an iterated function system $\Gamma=\IFS$ and a set of function $H=\{h_i:K\to\rn\}_{i=1}^d$ satisfying suitable conditions, we define a generalized gauge action on Kawjiwara-Watatani algebras $\CO_\Gamma$ and their Toeplitz extensions $\CT_\Gamma$. We then characterize the KMS states for this action. For each $\beta\in(0,\infty)$, there is a Ruelle operator $\CL_{H,\beta}$ and the existence of KMS states at inverse temperature $\beta$ is related to this operator. The critical inverse temperature $\beta_c$ is such that $\CL_{H,\beta_c}$ has spectral radius 1. If $\beta<\beta_c$, there are no KMS states on $\CO_\Gamma$ and $\CT_\Gamma$; if $\beta=\beta_c$, there is a unique KMS state on $\CO_\Gamma$ and $\CT_\Gamma$ which is given by the eigenmeasure of $\CL_{H,\beta_c}$; and if $\beta>\beta_c$, including $\beta=\infty$, the extreme points of the set of KMS states on $\CT_\Gamma$ are parametrized by the elements of $K$ and on $\CO_\Gamma$ by the set of branched points.
\end{abstract}

\maketitle

\section{Introduction}

Several examples of fractals are self-similar sets and can be built using an iterated function system \cite{B1,Ed1,F1}. In \cite{KW1}, Kajiwara and Watatani introduced C*-algebras associated with self-similar sets arising from iterated function systems. One of their main goals with their construction was to codify the structure of branched points of the iterated function system inside the algebra. Indeed, they showed that, under certain conditions, the branched points are reflected in the structure of the KMS states for the gauge action on their algebras \cite{IKW,KW2}.

For C*-algebras arising from dynamical systems, there is also an interest in studying generalized gauge actions \cite{E3,IK,KR,PWY}. In these works there were no branched points and they were able to prove the existence and uniqueness of KMS states by relying on a version of the Ruelle-Perron-Frobenius theorem.

The first main goal of this paper is to define generalized gauge actions on Kajiwara-Watatani algebras and study their KMS states. We rely on the Ruelle-Perron-Frobenius theorem for iterated function systems proved by Fan and Lau in \cite{FL}. In \cite{IKW,KW2}, they showed that if the iterated function system consists of $d$ maps, then there is a unique KMS state for inverse temperature $\beta=\log d$ which is given by the Hutchinson measure on the self-similar set; if $\beta > \log d$ then the extreme points of the set of KMS states are parametrized by the branched points; and if $\beta<\log d$, there are no KMS states. Because of the change of behaviour on the set of KMS states at $\log d$, we say that $\log d$ is a critical inverse temperature. For the generalized gauge action, we consider a family of Ruelle operators indexed by $\beta > 0$ (see Equation~\ref{eq:L}). For each $\beta$, we denote by $\rho(\beta)$ the spectral radius of the corresponding Ruelle operator. We show now that the inverse critical temperature $\beta_c$ is the unique $\beta$ satisfying $\rho(\beta_c)=1$. For $\beta=\beta_c$ there is also a unique KMS state which is given by the eigenmeasure of the Ruelle operator; for $\beta<\beta_c$ there are no KMS states; and for $\beta>\beta_c$ as before, the extreme points of the set of KMS states is again parametrized by the branched points (see Theorem~\ref{thm:KMS}).

Kajiwara-Watatani algebras were built from C*-correspondences and there is a Toeplitz version of their algebra. Although the general results of Laca-Neshveyev \cite{LN1} also deals with Toeplitz-Cuntz-Pimsner algebras, the analysis of KMS states on the Topelitz algebras of self-similar sets was not considered in \cite{IKW,KW2}. Our second main goal is then to describe KMS states for the generalized gauge action on the Topelitz algebras of iterated function systems. The same inverse critical temperature $\beta_c$ applies in the Toeplitz algebra. The main difference is that for $\beta>\beta_c$, the extreme points of the set of KMS states is parametrized not only of the branched points but all points of the self-similar set (see Theorem~\ref{thm:KMS}).

The structure of the paper is as follows. In Section \ref{sec:prelim}, we recall some basic definitions and results on iterated function systems, self-similar sets, Cuntz-Pimsner algebra, Kajiwara-Watatani algebras and KMS states on Cuntz-Pimsner algebras. In Section \ref{sec:opg}, we define a generalized gauge action on Kajiwara-Watatani algebras and give conditions for it to satisfy the hypothesis of the Laca-Neshevey theorem about KMS states on Cuntz-Pimsner algebras. Finally, we study KMS states for the generalized gauge action both on Kajiwara-Watani algebras and their Toeplitz version in Section \ref{sec:KMS}

\subsection*{Acknowledgements:} This paper contains revised and extended results of the author's PhD thesis that were previously unpublished. The author would like to thank his supervisors Artur O. Lopes, Jean Renault and Ruy Exel. The author would also like to thank Daniel Gonçalves for useful discussions.

\section{Preliminaries}\label{sec:prelim}

\subsection{Iterated function systems}

In this section, we review some definitions and results of the basic theory of iterated function systems and self-similar sets (see for instance \cite{B1,Ed1,F1}). Fix $(M,\rho)$ be a compact metric space.

\begin{definition}
We say that a function $\gamma :M\rightarrow M$ is a \emph{contraction} if there exists $c\in (0,1)$ such that $\rho (\gamma
(x),\gamma (y))\leq c\rho (x,y)$.
\end{definition}

\begin{definition}
An \emph{iterated function system} over $M$ is a finite set of continuous
functions $\Gamma=\left( \gamma_{i}:M\rightarrow M\right)_{i=1}^{d}$. We say that the iterated function system is \emph{contractive} if all functions are contractions.
\end{definition}

\begin{proposition}
Given a contractive iterated function system $\Gamma=\IFS$, there is a unique compact non-empty subset $K\ $of $M$ such that
\begin{equation}
K=\bigcup _{i=1}^{d}\gamma _{i}(K).  \label{eq:InvSet}
\end{equation}
\end{proposition}

The set $K$ as above is called the \emph{attractor} of the iterated function system and we say it is a \emph{self-similar set}.

Note that, because of \eqref{eq:InvSet}, the attractor is invariant by all $\gamma _{i}$ and we can restrict the iterated function system to its attractor. From now on, we assume that $M=K$.

Consider the set $\Sigma_d =\left\{ 1,\ldots ,d\right\} ^{\mathbb{N}}$ with the
product topology, $\sigma :\Sigma_d \rightarrow \Sigma_d $ the left shift and, for each $i=1,\ldots,d$, the function $\sigma _{i}:\Sigma_d \rightarrow \Sigma_d $ given by
\[
\sigma _{i}(i_{0},i_{1,}\ldots )=(i,i_{0},i_{1,}\ldots ).
\]%
This system is called the \emph{full shift}.

\begin{proposition}
\label{prop:CodeMap}Let $\Gamma=\IFS$ be a contractive iterated function system and $K$ its attractor. Then there is a continuous
surjection $F:\Sigma_d \rightarrow K$ such that $F\circ \sigma _{i}=\gamma_{i}\circ F$. This map is given by the formula
\[
F(i_{0},i_{1,}\ldots )=\lim_{n\rightarrow \infty }\gamma _{i_{0}}\circ\cdots \circ \gamma _{i_{n}}(x)
\]
for an arbitrary $x\in K$.
\end{proposition}

The following definitions are used in the study of C*-algebras associated with self-similar sets and their KMS states (see \cite{KW1,Ma1}).

\begin{definition}
	Let $\Gamma=\IFS$ be an iterated function system, we define the following sets 
	$$B\IFS:=\{x\in K\mid \exists y\in K\ \exists i\neq j:x=\gamma_i(y)=\gamma_j(y)\};$$
	$$C\IFS:=\{y\in K\mid \exists i\neq j:\gamma_i(y)=\gamma_j(y)\}.$$
	$$I(x):=\{i\in \{1,...,d\};\exists y\in K:x=\gamma _{i}(y)\}.$$
	We call the points of $B\IFS$ \emph{branched points} and the points of $C\IFS$ \emph{branched values}. And we say that $\Gamma$ satisfies the \emph{finite branch condition} if $C\IFS$ is finite.
\end{definition}

\begin{definition}
	Let $\Gamma=\IFS$ be an iterated function system. For $x\in K$ and $n\in\mathbb{N}$, we set the \emph{$n$-th orbit of $x$} by
	$$O_n(x)=\{\gamma_{i_1}\circ\cdots\circ\gamma_{i_n}(x)\in K:1\leq i_1,\ldots,i_n\leq d\},$$
	and define the \emph{orbit of $x$} by $O(x)=\cup_{n=0}^{\infty}O_n(x)$. We will say that the iterated function system $\Gamma$ satisfies the \emph{escape condition} if for every $x\in K$, there exists $y\in O(x)$ such that $O(y)\cap C\IFS=\emptyset$.	
\end{definition}

\subsection{Cuntz-Pimsner algebras}
	
	We briefly recall the key elements for the construction of Cuntz-Pimsner algebras (\cite{K1,P1}) that will be used throughout the paper. For that, fix $A$ a C*-algebra.
	\begin{definition}
		A \emph{(right) Hilbert C*-module over $A$} is a (right-)$A$-module $X$ with a sesquilinear map $\left<\ ,\ \right>:X\times X\to A$ such that:
		\begin{enumerate}[(i)]
			\item $\left<\xi,\eta a\right>=\left<\xi,\eta\right>a$;
			\item $(\left<\xi,\eta\right>)^{\ast}=\left<\eta,\xi\right>$;
			\item $\left<\xi,\xi\right>\geq 0$;
			\item $X$ is complete with respect to the norm $\|\xi\|_2=\|\left<\xi,\xi\right>\|^{1/2}$
		\end{enumerate}
		for $a\in A$ and $\xi,\eta\in X$.
		We say that $X$ is \emph{full} if $\left<X,X\right>$ is dense in $A$.
	\end{definition}
		Let $X$ be a Hilbert C*-module and denote by $\CL(X)$ the space of adjointable operators in $X$. We note that $\CL(X)$ is a C*-algebra.	
	For $\xi,\eta\in X$ we define an operator $\theta_{\xi,\eta}:X\to X$ by $\theta_{\xi,\eta}(\zeta)=\xi\left<\eta,\zeta\right>$. This is an adjointable operator and we denote by $\mathcal{K}(X)$ the closed subspace of $\CL(X)$ generated by all $\theta_{\xi,\eta}$.
	\begin{definition}
		A \emph{C*-correspondence} over $A$ is a Hilbert C*-module $X$ together with a C*-homomorphism $\phi:A\to\CL(X)$.
	\end{definition}
	Let ($X$,$\phi$) be a C*-correspondence over $A$ and, for simplicity, suppose that $\phi$ is faithful. We denote by $J_X$ the ideal $\phi^{-1}(\mathcal{K}(X))$.
	\begin{definition}\label{def:PimsnerCovRep}
		A pair $(\iota,\psi)$ of maps $\iota:A\to B$, $\psi:X\to B$, where $B$ is a C*-algebra and $\iota$ a C*-homomorphism, is said to be a \emph{representation of $X$} if:
		\begin{itemize}
			\item[(i)] $\psi(\phi(a)\xi b)=\iota(a)\psi(\xi)\iota(b)$;
			\item[(ii)] $\psi(\xi)^{\ast}\psi(\eta)=\iota(\left<\xi,\eta\right>)$;
		\end{itemize}
		for $a,b\in A$, $\xi,\eta\in X$ and $c\in J_X$.
		If moreover
		\begin{itemize}
			\item[(iii)] $(\psi,\iota)^{(1)}(\phi(c))=\iota(c)$ where $(\psi,\iota)^{(1)}:\mathcal{K}(X)\to B$ is given by $(\psi,\iota)^{(1)}(\theta_{\xi,\eta})=\psi(\xi)\psi(\eta)^{\ast}$,
		\end{itemize}
		for all $c\in J_X$, we say that $(\iota,\psi)$ is a \emph{covariant representation of $X$}
	\end{definition}

	\begin{definition}
		The \emph{Toeplitz-Cuntz-Pimsner algebra} of $X$, denoted by $\CT_X$, is the universal C*-algebra with respect to representations of $X$. Similarly, the \emph{Cuntz-Pimsner algebra} of $X$, denoted by $\CO_X$, is the universal C*-algebra with respect to covariant representations of $X$.
	\end{definition}

\subsection{Kajiwara-Watatani algebras}\label{sec:KWalg}

Let $\Gamma=\IFS$ be a contractive iterated function system and $K$ its attractor. We recall the C*-correspondence defined in \cite{KW1}. There, they ask for the maps to be proper contractions, but for the construction of the algebra, as long as we have a self-similar set, we do not even need contractions (see \cite{Mu1}). We let $A=C(K)$, 
$X=C(\mathcal{G})$ where%
$$\mathcal{G}=\bigcup _{i=1}^{d}\mathcal{G}_i$$
with
$$\mathcal{G}_i=\left\{ (x,y)\in K\times K:x=\gamma _{i}(y)\right\}$$
being the cograph of $\gamma_i$ in the terminology of \cite{KW1}. The structure of C*-correspondence is given by%
\[
(\phi(a)\xi b)(x,y)=a(x)\xi(x,y)b(y) 
\]%
and%
\begin{equation}\label{eq:inner.product}
\left\langle \xi,\eta\right\rangle _{A}(y)=\sum_{i=1}^{d}\overline{\xi(\gamma
_{i}(y),y)}\eta(\gamma _{i}(y),y) 
\end{equation}
for $a,b\in A$ and $\xi,\eta\in X$.

	\begin{proposition}\cite[Proposition 2.1]{KW1}\label{prop:BasicResultsCorr}
		$X=(C(\mathcal{G}),\phi)$ is a full C*-correspondence over $A=C(K)$ and $\phi:A\to\CL(X)$ is faithful and unital. Moreover, the Hilbert module norm is equivalent to the sup norm in $C(\mathcal{G})$.
	\end{proposition}

\begin{definition}
		The \emph{Kajiwara-Watatani algebra} $\mathcal{O}_{\Gamma}$ associated with $\Gamma$ is the Cuntz-Pimsner algebra associated with the C*-correspondence defined above. The \emph{Toeplitz algebra $\CT_\Gamma$ associated with $\Gamma$} is the corresponding Toeplitz-Cuntz-Pimsner algebra.
\end{definition}

\begin{lemma}\cite[Lemma~2.8]{KW2}\label{lemma:FinBranchIdeal}
If $\Gamma$ satisfies the finite branch condition then $J_X=\{a\in A=C(K);a~\mathrm{vanishes~on~}B(\gamma
_{1},...,\gamma _{d})\}$.
\end{lemma}

\subsection{KMS states on Cuntz-Pimsner algebras}

We review some results of \cite{LN1} that will be used to describe KMS states on Kajiwara-Watatni algebras. For the basic definitions on KMS states, we refer the reader to \cite{Ped}.

Let $(X,\phi)$ be a full C*-correspondence over $A$ which is non-degenerated. To define a quasi-free dynamics we need a one-parameter group of automorphisms $\delta=\{\delta_t\}_{t\in\rn}$ of $A$ and a one-parameter group os isometries $\upsilon=\{\upsilon_t\}_{t\in\rn}$ of $X$ such that $\upsilon_t(\phi(a)\xi)=\phi(\delta_t(a))\upsilon_t(\xi)$ and $\left<\upsilon_t(\xi),\upsilon_t(\eta)\right>=\delta_t(\left<\xi,\eta\right>)$ for all $a\in A$ and $\xi,\eta\in X$. By the universal property of $\CT_X$, we get a one-parameter group of automorphisms $\{\sigma_t\}_{t\in\rn}$ that can be restricted to $\CO_X$. In our case, we will suppose that $\delta_t(a)=a$ and $\upsilon_t(\xi)=e^{-itD}\xi$ where $D$ is a self-adjoint element of $\CL(X)$.

\begin{definition}
	For $\tau$ a tracial state in $A$ and $T\in\CL(X)$, define
	$$\Tr_{\tau}(T)=\lim_{k\to\infty}\sum_{\xi\in I_k}{\omega(\left<\xi,T\xi\right>)}$$
	where $\{e_k=\sum_{\xi\in I_k}\theta_{\xi,\xi}\}$ is an approximate unit of $\mathcal{K}(X)$.
\end{definition}

It is shown in \cite[Theorem~1.1]{LN1} that $\Tr_{\tau}(T)$ does not depend on the choice of approximate unit and that it can be extended to a positive linear functional on a suitable space.

Before stating Laca-Neshveyev's theorem, we briefly recall the Arveson spectrum \cite{Arv} (see also \cite{Ma1}). For $f\in L^1(\rn)$, we let $\widehat{f}$ denote its Fourier transform. For the one-parameter group os isometries $\upsilon=\{\upsilon_t\}_{t\in\rn}$, we let $\pi(f)\in\CL(X)$ be the operator given by
\begin{equation}\label{eq:pif}
\pi(f)\xi=\int_{\rn}f(t)\upsilon_{-t}(\xi)dt.
\end{equation}
The Arveson spectrum of $\xi\in X$ with respect to $\upsilon$ is given by
\[\Sp_{\upsilon}(\xi)=\{\theta\in\rn\mid \widehat{f}(\theta)=0\text{ for all }f\text{ such that }\pi(f)\xi=0\}.\]

\begin{definition}
	We say that the one-parameter group os isometries $\upsilon=\{\upsilon_t\}_{t\in\rn}$ satisfies \emph{the positive energy condition} if the set $\{\xi\in X\mid\Sp_{\upsilon}(\xi)\scj(0,\infty)\}$ is dense in $X$.
\end{definition}

Since we will be only interested in the C*-correspondence given in Section \ref{sec:KWalg}, we restrict to the case where $A$ is commutative.

\begin{theorem}\cite[Theorems 2.1, 2.2 and 2.5]{LN1}\label{thm:LN}
	Let $X$ be a C*-correspondence over a commutative C*-algebra $A$ and $D$ an self-adjoint operator on $X$. Suppose that $\upsilon=\{\upsilon_t\}_{t\in\rn}$ given by $\upsilon_t=e^{-itD}$ satisfies the positive energy condition and let $\sigma$ be the corresponding one-parameter group of automorphisms. For every $\beta\in(0,\infty]$, there is a bijection between the set $\varphi$ of $(\sigma,\beta)$-KMS states on $\CT_X$ and the set of states $\tau$ on $A$ satisfying
	\begin{itemize}
		\item[(K2)] $\Tr_{\tau}(ae^{-\beta D})\leq\tau(a)$ for all $a\in A^+$,
	\end{itemize}
	where $\Tr_{\tau}(ae^{-\infty D})=0$.
	Moreover, there is a bijection between the set of $(\sigma,\beta)$-KMS states on $\CO_X$ and the set of states on $A$ satisfying (K2) and
	\begin{itemize}
		\item[(K1)] $\Tr_{\tau}(ae^{-\beta D})=\tau(a)$ for all $a\in J_X$.
	\end{itemize}
	Both bijections are given by $\tau=\varphi|_A$.
\end{theorem}

The choice for the names (K1) and (K2) on the conditions above is to be consistent with \cite{IKW}.

\begin{definition}\cite[Definition 2.3]{LN1}\label{def:fininfgeral}
	In the condition of Theorem~\ref{thm:LN}, let $\varphi$ be a $(\sigma,\beta)$-KMS state and $\tau=\varphi|_A$. For $\tau'$ a state on $A$, we let $\CF(\tau')=\Tr_{\tau'}(\cdot e^{-\beta D})$. We say that $\varphi$ is of \emph{finite type} if there exists a finite trace $\tau_0$ such that $\tau=\sum_{n=0}^{\infty}\CF^{n}(\tau_0)$ in the weak*-topology. We say that $\varphi$ is of \emph{infinite type} if $\tau=\CF(\tau)$.
\end{definition}

\begin{proposition}\cite[Proposition 2.4]{LN1}
	In the condition of Theorem~\ref{thm:LN}, let $\varphi$ be a $(\sigma,\beta)$-KMS state. Then there exists a unique convex combination $\varphi=\lambda\phi_1+(1-\lambda)\varphi_2$ where $\varphi_1$ is a $(\sigma,\beta)$-KMS state of finite type and $\varphi_2$ is a $(\sigma,\beta)$-KMS state of infinite type.
\end{proposition}

It follows from this proposition that to study the set of $(\sigma,\beta)$-KMS states it is enough to study KMS states of finite and infinite types. Also, there are no KMS-states of infinite type for $\beta=\infty$.

\section{The generalized gauge action}\label{sec:opg}

The goal of this section is to define a one-parameter group of automorphisms that generalizes the gauge action on Kajiwara-Watatni algebras. We keep the notation of Section \ref{sec:KWalg}, that is, for an iterated function system $\Gamma=\IFS$ with attractor $K$, we let $\CG$ be the union of the cographs, $A=C(K)$ and $X=C(\CG)$.

\begin{proposition}\label{prop:infinitesimal.generator}
	Let $\Gamma=\IFS$ be a contractive iterated function system satisfying the finite branch condition and let $\{h_k\}_{k=1}^d$ be a set of strictly positive continuous functions on $K$ that are compatible with the branched points, in the sense that for all $x\in B\IFS$ and all $k,l\in I(x)$, we have that $h_k(x)=h_l(x)$. For each $\xi\in X$, the function $D\xi:\CG\to\cn$ given by
	\[D\xi(\gamma_k(y),y)=\ln(h_k(\gamma_k(y)))\xi(\gamma_k(y),y)\]
	is well-defined and continuous. Moreover the map $D:X\to X$ that sends $\xi\in X$ to $D\xi$ is an self-adjoint element of $\CL(X)$.
\end{proposition}

\begin{proof}
	Fix $\xi\in X$. To see that $D\xi$ is well-defined, we take $(x,y)\in\CG$. If $x\notin B\IFS$, then there exists a unique $k=1,\ldots,d$ such that $(x,y)=(\gamma_k(y),y)$ so that $D\xi(x,y)$ is uniquely determined. And if $x\in B\IFS$, by the hypothesis on the family $\{h_k\}_{k=1}^d$ we have that $\ln(h_k(x))\xi(x,y)=\ln(h_l(x))\xi(x,y)$, whenever $x=\gamma_k(x)=\gamma_l(x)$.
	
	Because each $\gamma_k$, $k=1,\ldots,d$, is a continuous function, if $(x,y)\in\CG$ is such that $x\notin B\IFS$, then there exists an open neighbourhood $U$ of $(x,y)$ in $\CG$ such that $U\cap \CG_k\neq\emptyset$ for a unique $k$. The continuity of $D\xi$ at $(x,y)$ follows immediately from the continuity of $\xi$ and $h_k$. Suppose now that $(x,y)$ is such that $x\in B\IFS$ and let $I=\{k\in\{1,\ldots,d\}\mid x=\gamma_k(y)\}$. Again, we can find a neighbourhood $U$ of $(x,y)$ such that $U\cap \CG_k\neq\emptyset$ if and only if $k\in I$. Because $I$ is finite and each $h_k$ is continuous, it is straightforward to show that $D\xi$ is continuous at $(x,y)$.
	
	Clearly $D$ is a linear operator on $X$ which is self-adjoint because $\ln(h_k)$ is a real function for each $k=1,\ldots,d$ and due to the definition of the inner product in \eqref{eq:inner.product}.
\end{proof}

\begin{corollary}\label{cor:action}
	In the conditions of Proposition \ref{prop:infinitesimal.generator}, the family $\{\upsilon_t\}_{t\in\rn}$ given by $\upsilon_t=e^{itD}$ is a one-parameter group of isometries on $X$. Moreover $\upsilon_t\phi(a)=\phi(a)\upsilon_t$ for all $a\in A$ and $t\in\rn$, and $\upsilon_t\in\CL(X)$ for all $t\in\rn$.
\end{corollary}

\begin{proof}
	The first part follows from the fact that $D$ is a self-adjoint operator in $\CL(X)$. For the second part, let $a,b\in A$, $\xi\in X$, $y\in K$, $k\in\{1,\ldots,d\}$ and $t\in\rn$. Then
	\[\upsilon_t\phi(a)(\xi b)(\gamma_k(y),y)=h_k^{it}(\gamma_k(y))a(\gamma_k(y))\xi(\gamma_k(y),y)b(y)=\phi(a)(\upsilon_t(\xi)b)(\gamma_k(y),y).\]
	The above equality proves both that $\upsilon_t\phi(a)=\phi(a)\upsilon_t$ for all $a\in A$ and $t\in\rn$, and $\upsilon_t\in\CL(X)$ for all $t\in\rn$.
\end{proof}

\begin{lemma}\label{lem:positive.energy}
	In the condition of Proposition \ref{prop:infinitesimal.generator}, suppose moreover that $h_k>1$ for all $k=1,\ldots,d$. Then the one-parameter group of $\{\upsilon_t\}_{t\in\rn}$ given by $\upsilon_t=e^{itD}$ satisfies the positive energy condition.
\end{lemma}

\begin{proof}
	Let $f\in L^1(\rn)$, $\xi\in X$ and $(\gamma_k(y),y)\in\CG$. Evaluating \eqref{eq:pif}, we get
	\begin{align*}
		\pi(f)\xi(\gamma_k(y),y) &= \int_{\rn} f(t)\upsilon_{-t}(\xi)(\gamma_k(y),y)dt\\
		&=\int_{\rn}f(t)e^{-it\ln(h_k(\gamma_k(y))}dt\xi(\gamma_k(y),y) \\
		&=\widehat{f}\left(\frac{1}{2\pi}\ln(h_k(\gamma_k(y))\right)\xi(\gamma_k(y),y).
	\end{align*}
    Since $h_k$ are continuous functions on $K$ such that $h_k>1$ and that $K$ is compact, there exists $c>0$ such that $\frac{1}{2\pi}\ln(h_k(\gamma_k(y))\geq c$ for all $k=1,\ldots,d$ and all $y\in K$. This implies that if $\theta\leq 0$, we can find $f\in L^1(\rn)$ such that $\pi(f)\xi=0$ and $\widehat{f}(\theta)\neq 0$. Hence $\Sp_{\upsilon}(\xi)\scj (0,\infty)$.
\end{proof}

\begin{remark}
	The condition that the potentials are greater than one also appears in other works studying KMS states on C*-algebras (see for instance \cite{E3,KR}).
\end{remark}

\begin{definition}
	The one-parameter group of automorphisms of both $\CO_\Gamma$ and $\CT_\Gamma$ given by the one-parameter groups of isometries in Corollary~\ref{cor:action} will be called the \emph{generalized gauge action given by $H$}.
\end{definition}

\section{KMS states on Kajiwara-Watatani algebras}\label{sec:KMS}

In this section, we see that several techniques used in \cite{IKW,KW2} to study the KMS states for the gauge action can be extended to study the KMS states for the generalized gauge action defined in Section \ref{sec:opg}.

Fix $\Gamma=\IFS$ an iterated function system satisfying the finite branch condition such that $d\geq 2$ and $K$ its attractor. Let $H=\hpot$ be a family in $C(K)$ compatible with the branches as in Proposition \ref{prop:infinitesimal.generator} and such that $h_j>1$ for all $j=1,\ldots,d$. We let $D$ and $\upsilon$ as in Section \ref{sec:opg}, and $\sigma$ the corresponding generalized gauge action. Moreover, given a point $(x,y)\in K\times K$, we set
\[e(x,y)=\#\{j\in\{1,\ldots,d\}\mid\gamma_j(y)=x\}.\]

For each $\beta\in(0,\infty)$ and $a\in C(K)$, we define complex functions on $K$, $\CL_{H,\beta}(a)$ and $\CS_{H,\beta}(a)$, by
\begin{equation}\label{eq:L}
	\CL_{H,\beta}(a)(y)=\sum_{j=1}^d h_j^{-\beta}(\gamma_j(y))a(\gamma_j(y))
\end{equation}
and
\begin{equation}\label{eq:S}
	\CS_{H,\beta}(a)(y)=\sum_{j=1}^d\frac{1}{e(\gamma_j(y),y)} h_j^{-\beta}(\gamma_j(y))a(\gamma_j(y)).
\end{equation}
We notice that $\CL_{H,\beta}(a)$ is continuous but $\CS_{H,\beta}(a)$ may not necessarily be continuous. Moreover the map $a\in C(K)\mapsto \CL_{H,\beta}(a)\in C(K)$ is a positive linear operator, which we call a \emph{Ruelle operator}. 

In what follows, because of the Riesz-Markov-Kakutani representation theorem, elements of $C(K)^*$ will be used as complex regular Borel measures on $K$ and vice-versa, whenever convenient.

\begin{lemma}\label{lem:compute.tr}
	Let $\beta>0$, $a\in C(K)_+$, $\tau\in C(K)_+^*$. Then
	\[\Tr_{\tau}(ae^{-\beta D})=\int_K \CS_{H,\beta}(a)d\tau.\]
\end{lemma}

\begin{proof}
	The proof is analogous to that of \cite[Theorem 4.2]{IKW} and we just point out two key differences. There, for $y\in K$, they define $\CG_y=\{x\in K\mid \exists j\in\{1,\ldots,d\}\text{ such that }x=\gamma_j(y)\}$. Also, for $a\in C(K)$ and $y\in K$, they set
	\[\widetilde{a}(y)=\sum_{x\in\CG_y}a(x).\]
	
	Because the family $H$ is compatible with the branches, we can define a function $h:\CG\to\cn$ by $h(x,y)=h_j(x)$, where  $j\in\{1,\ldots,d\}$ is such that $x=\gamma_j(y)$. This way, we can rewrite Equation \eqref{eq:S} as
	\[\CS_{H,\beta}(a)(y)=\sum_{x\in\CG_y}h^{-\beta}(x,y)a(x).\]
	So $\CS_{H,\beta}(a)$ plays the role of $\widetilde{a}$ there.
	
	Moreover, if $\xi\in X$, then
	\begin{align*}
		\left<\xi,ae^{-\beta D}\xi\right>(y)&=\sum_{j=1}^d|\xi(\gamma_j(y),y)|^2a(\gamma_j(y))h^{-\beta}_j(\gamma_j(y))\\
		&=\sum_{x\in\CG_y}e(x,y)|\xi(x,y)|^2a(x)h^{-\beta}(x,y).
	\end{align*}
	Although the $h$ above depends on $(x,y)$, this will not hinder the computation done in the proof of \cite[Theorem~4.2]{IKW}.
\end{proof}

\begin{lemma}\label{lem:extension}
	Let $\beta>0$ and $\tau\in C(K)^*_+$. The map $a\in C(K)_+\mapsto \Tr_{\tau}(ae^{-\beta D})$ extends to an element of $C(K)^*_+$.	 
\end{lemma}

\begin{proof}
	It follows from \cite[Theorem~1.1(ii)]{LN1}, observing that for $a\in C(K)$, using Lemma \ref{lem:compute.tr}, we get $\Tr_{\tau}(ae^{-\beta D})\leq d\max_j\{\|h_j^{-\beta}\|\}\|a\|\|\tau\|$.
\end{proof}

\begin{definition}
	We define the function $\CF_{H,\beta}:C(K)^*_+\to C(K)^*_+$ as the extension of $a\in C(K)_+\mapsto \Tr_{\tau}(ae^{-\beta D})$ given by Lemma \ref{lem:extension}. For $\beta=\infty$, we have that $\CF_{H,\beta}=0$.
\end{definition}

In order to describe KMS states for our situation, we adapt Definition~\ref{def:fininfgeral}.

\begin{definition}
	Let $\tau$ be a state on $A$. We say that $\tau$ is of \emph{finite type} with respect to $(H,\beta)$ if if there exists a finite trace $\tau_0$ such that $\tau=\sum_{n=0}^{\infty}\CF_{H,\beta}^{n}(\tau_0)$ in the weak*-topology. We say that $\tau$ is of \emph{infinite type} with respect to $(H,\beta)$ if $\tau=\CF_{H,\beta}(\tau)$.
\end{definition}

We start with a few lemmas comparing $\CL_{H,\beta}$ with $\CS_{H,\beta}$ and $\CL_{H,\beta}^*$ with $\CF_{H,\beta}$.

\begin{lemma}\label{lem:compare S and L}
	Let $a\in C(K)$.
	\begin{enumerate}[(i)]
		\item \label{lem:S=L in K-C} $\CL_{H,\beta}(a)(y)=\CS_{H,\beta}(a)(y)$ for all $y\in K\setminus C\IFS$. 
		\item \label{lem:S=L in J_X}
		If $a\in J_X$, then $\CL_{H,\beta}(a)=\CS_{H,\beta}(a)$.
		\item \label{lem:S<L positive}
		If $a\in C(K)_+$, then $\CS_{H,\beta}(a)\leq\CL_{H,\beta}(a)$.
	\end{enumerate}
	
\end{lemma}

\begin{proof}
	\eqref{lem:S=L in K-C} We compare the expression in Equations \eqref{eq:L} and \eqref{eq:S}. If $y\in K\setminus C\IFS$, since $e(\gamma_j(y),y)=1$ for all $j=1,\ldots,d$, we see that, in this case, $\CL_{H,\beta}(a)(y)=\CS_{H,\beta}(a)(y)$.
	
	\eqref{lem:S=L in J_X} Let $a\in J_X$. \eqref{lem:S=L in K-C}, it suffice to show that $\CL_{H,\beta}(a)(y)=\CS_{H,\beta}(a)(y)$ for all $y\in C\IFS$. In this case, for $k\in\{1,\ldots,d\}$, if $\exists j\neq k$ such that $\gamma_k(y)=\gamma_j(y)$ then $\gamma_k(y)\in B(\gamma_1,\ldots,\gamma_n)$ and in this case $a(\gamma_k(y))=0$ by Lemma~\ref{lemma:FinBranchIdeal}. If there is no such $j$, then $e(\gamma_k(y),y)=1$. Again, comparing Equations \eqref{eq:L} and \eqref{eq:S}, we see that the equality is also true for $y\in C(\gamma_1,\ldots,\gamma_n)$.
	
	\eqref{lem:S<L positive} This follows from Equations \eqref{eq:L} and \eqref{eq:S}, observing that $e(x,y)\geq 1$ for all $(x,y)\in\CG$ and that we are dealing with positive functions.
\end{proof}

\begin{lemma}\label{lem:compare F and L} Let $\tau\in C(K)^{*}_+$.
	\begin{enumerate}[(i)]
		\item\label{lem:F < L for positive} For every $a\in C(K)^{+}$ and $n\in\nn^*$, we have that
		\[\CF_{H,\beta}^n(\tau)(a)\leq(\CL_{H,\beta}^{*})^n(\tau)(a).\]
		\item\label{lem:F=L if tau(C)=0} If $\tau(C\IFS)=0$, then for every $n\in\nn^*$, we have that $\CF_{H,\beta}^n=(\CL_{H,\beta}^*)^n$.
	\end{enumerate}
\end{lemma}

\begin{proof}
	\eqref{lem:F < L for positive} We prove by induction on $n$. If $n=1$, then, by Lemmas \ref{lem:compute.tr} and \ref{lem:compare S and L}\eqref{lem:S<L positive}, we have that
	\[\CF_{H,\beta}(\tau)(a)=\int_K \CS_{H,\beta}(a)d\tau\leq\int_K \CL_{H,\beta}(a)d\tau=\CL_{H,\beta}^*(\tau)(a).\]
	
	Now let $n\in\nn^*$ and suppose that $\CF_{H,\beta}^n(\tau)(a)\leq(\CL_{H,\beta}^{*})^n(\tau)(a)$. Then, using the base case on $\CF_{H,\beta}^n(\tau)$ and the fact that $\CL_{H,\beta}$ is a positive operator, we have that
	\[\CF_{H,\beta}^{n+1}(\tau)(a)\leq\CL^*_{H,\beta}(\CF_{H,\beta}^{n}(\tau))(a)\leq(\CL^*_{H,\beta})^{n+1}(\tau)(a),\]
	proving the induction step.
	
	\eqref{lem:F=L if tau(C)=0} The proof is analogous using Lemma \ref{lem:compare S and L}\eqref{lem:S=L in K-C} instead of \ref{lem:compare S and L}\eqref{lem:S<L positive}.
\end{proof}

The following proposition connects fixed points of $\CL_{H,\beta}^*$ with states of infinite type with respect to $(H,\beta)$. In particular, by Theorem~\ref{thm:LN} and the following proposition, a fixed point $\CL_{H,\beta}^*$ of point always give rise to a KMS state.

\begin{proposition}\label{prop:inf=eigen}
	Let $\tau\in C(K)^{*}_+$ be a state, if $\CL_{H,\beta}^{*}(\tau)=\tau$ then $\tau$ satisfies (K1) and (K2) of Theorem \ref{thm:LN}. Moreover, if $\tau(C(\gamma_1,\ldots,\gamma_n))=0$, then $\CL_{H,\beta}^{*}(\tau)=\tau$ if and only if $\tau$ is of infinite type with respect to $(H,\beta)$.
\end{proposition}

\begin{proof}
	Take $a\in J_X$. By Lemmas \ref{lem:compute.tr} and \ref{lem:compare S and L}\eqref{lem:S=L in J_X}, we have that
	\[\CF_{H,\beta}(\tau)(a)=\int_K \CS_{H,\beta}(a)d\tau=\int_K \CL_{H,\beta}(a)d\tau=\CL_{H,\beta}^*(\tau)(a)=\tau(a)\]
	which proves (K1). Now, take $a\in A^{+}$, then, by Lemma \ref{lem:compare F and L}\eqref{lem:F < L for positive},
	\[\CF_{H,\beta}(\tau)(a)\leq\CL_{H,\beta}^*(\tau)(a)=\tau(a)\]
	which proves (K2).
	
	If $\tau(C(\gamma_1,\ldots,\gamma_n))=0$ then, by Lemma \ref{lem:compare S and L}\eqref{lem:S=L in K-C},
	\[\CF_{H,\beta}(\tau)(a)=\int_K \CS_{H,\beta}(a)d\tau=\int_K \CL_{H,\beta}(a)d\tau=\CL_{H,\beta}^*(\tau)(a)\]
	for all $a\in C(K)$, and the equivalence between $\CL_{H,\beta}^{*}(\tau)=\tau$ and $\tau$ being of infinite type follows.
\end{proof}

Now, let us restrict our attention to a certain class of functions for which we have a version of the Ruelle-Perron-Frobenius theorem. We will the results of \cite{FL}. First, we recall the definition of Dini-continuity.

\begin{definition}
	For a function $h:K\to\mathbb{R}$ we define the modulus of continuity by $\omega(h,t)=\sup\{|h(x)-h(y)|:d(x,y)\leq t\}$. And we say that $h$ is \emph{Dini-continuous} if
	\[\int_0^1\frac{\omega(h,t)}{t}dt<\infty.\]
\end{definition}

For each $\beta\in\rn$, we let $\rho(\beta)$ be the spectral radius of $\CL_{H,\beta}$. Since $\CL_{H,\beta}$ is a positive operator, so is $\CL_{H,\beta}^{*}$ and their spectral radius coincide. We also have the following formula
\[\rho(\beta)=\lim_{n\to\infty}\|\CL_{H,\beta}^n\|^{1/n}=\lim_{n\to\infty}\|\CL_{H,\beta}^n(1)\|^{1/n}.\]

\begin{theorem}\label{thm:RPF}\cite[Theorem 1.1]{FL}
	Suppose that $\log h_j$ is Dini-continuous for every $j=1,\ldots,d$, then for each $\beta\in\rn$ there is a unique positive function $k_{\beta}\in C(K)$ and a unique state $\tau_{\beta}\in C(K)^{*}$ such that
	$$\CL_{H,\beta}(k_{\beta})=\rho(\beta)k_{\beta},\qquad\CL_{H,\beta}^{*}(\tau_{\beta})=\rho(\beta)\tau_{\beta},\qquad\tau_{\beta}(k_{\beta})=1.$$
	Moreover, for every $a\in C(K)$, $\rho(\beta)^{-n}\CL_{H,\beta}^n(a)$ converges uniformly to $\tau_\beta(a)k_{\beta}$ and for every state $\theta\in C(K)^{*}$, $\rho(\beta)^{-n}(\CL_{H,\beta}^{*})^n(\theta)$ converges to $\theta(k_\beta)\tau_{\beta}$ in the weak*-topology.
\end{theorem}

\begin{proposition}\label{prop:analytic}\cite[Proposition 1.4]{FL}
	Suppose that $\log h_j$ is Dini-continuous for every $j=1,\ldots,d$, then the real function $\beta\mapsto \log \rho(\beta)$ is analytic.
\end{proposition}

\begin{corollary}\label{cor:sp.rad}
	Suppose that $\log h_j$ is Dini-continuous for every $j=1,\ldots,d$, then the real function $\rho$ that maps $\beta$ to $\rho(\beta)$ is strictly decreasing. Moreover, $\rho(\beta_c)=1$ for a unique $\beta_c>0$
\end{corollary}

\begin{proof}
	For $\beta_1,\beta_2\in\rn$ such that $\beta_1<\beta_2$ and for each $n\in\nn$, we have that
	\[\|\CL_{H,\beta_1}^n\|=\|\CL_{H,\beta_1}^n(1)\|\geq\|\CL_{H,\beta_2}^n(1)\|=\|\CL_{H,\beta_2}^n\|.\]
	Hence $\rho(\beta_1)\geq\rho(\beta_2)$. Also $\lim_{\beta\to\infty} \rho(\beta)=0$, so that $\rho$ is not constant. That $\rho$ is strictly decreasing then follows from Proposition \ref{prop:analytic}.
	
	Notice that $\rho(0)=d$. Since we are assuming $d\geq 2$, if follows from the first part that $\rho(\beta)=1$ for some $\beta>0$, which unique because $\rho$ is strictly decreasing.
\end{proof}

\begin{definition}
	Suppose that $\log h_j$ is Dini-continuous for every $j=1,\ldots,d$. We call the unique $\beta_c$ such that $\rho(\beta_c)=1$, given by Corollary \ref{cor:sp.rad}, the \emph{critical inverse temperature} for $H$.
\end{definition}

Now, let us go back to the study of KMS states of the generalized gauge action $\sigma$ on $\CO_\Gamma$ and $\CT_\Gamma$. We denote by $K_{\beta}(\CT_\Gamma)$ the set of $(\sigma,\beta)$-KMS states on $\CT_\Gamma$,  $K_{\beta}(\CT_\Gamma)_f$ the subset of KMS states of finite type and $K_{\beta}(\CT_\Gamma)_i$ the subset of KMS states of infinite type. Analogously we define $K_{\beta}(\CO_\Gamma)$, $K_{\beta}(\CO_\Gamma)_f$ and $K_{\beta}(\CO_\Gamma)_i$. Due to Lemma~\ref{lem:positive.energy}, we can use Theorem~\ref{thm:LN} to describe these sets.

From now on, we will assume that $\log h_j$ is Dini-continuous for every $j=1,\ldots,d$. We will divide our analysis for $\beta\in (0,\beta_c)$, $\beta=\beta_c$ and $\beta\in(\beta_c,\infty)$. For $\beta>0$, the three conditions are equivalent respectively to $\rho(\beta)>1$, $\rho(\beta)=1$ and $\rho(\beta)<1$, which we will use interchangeably. We also consider the case $\beta=\infty$, for which the analysis is similar to that of $\beta\in(\beta_c,\infty)$, even though we do not have a Ruelle operator for $\beta=\infty$.

\begin{lemma}\label{lem:betagrande}
	Let $\beta_c\in(\beta,\infty]$. Then
	\begin{enumerate}[(i)]
		\item\label{i:noinf} there are no states on $A$ that are of infinite type with respect to $(H,\beta)$;
		\item\label{i:allfin} for every state $\tau_0$ on $C(K)$, we have that $\sum_{n=0}^{\infty}\mathcal{F}_{H,\beta}^n(\tau_0)$ converges in the weak*-topology to an element of $C(K)^*$. Moreover, the map that sends $\tau_0$ to $\sum_{n=0}^{\infty}\mathcal{F}_{H,\beta}^n(\tau_0)$ preserves convex combinations.
	\end{enumerate}
\end{lemma}

\begin{proof}
	\eqref{i:noinf} In general, there are not states of infinite type for $\beta=\infty$. Let $\beta\in(\beta_c,\infty)$ and suppose that $\tau$ is a state of infinite type with respect to $(H,\beta)$, then, by Lemma \ref{lem:compare F and L}\eqref{lem:F < L for positive},
	$$1=\tau(1)=\CF_{H,\beta}^n(\tau)(1)\leq(\CL_{H,\beta}^{*})^n(\tau)(1)=|(\CL_{H,\beta}^{*})^n(\tau)(1)|\leq\|(\CL_{H,\beta}^{*})^n\|\|\tau\|\|1\|=\|(\CL_{H,\beta}^{*})^n\|.$$
	Hence
	$$1\leq\lim_{n\to\infty}\|(\CL_{H,\beta}^{*})^n\|^{1/n}=\rho(\beta)<1$$
	which is a contradiction.
	
	\eqref{i:allfin} The result is trivial for $\beta=\infty$, so let $\beta\in(\beta_c,\infty)$. Take a non-zero element $a\in C(K)^{+}$, then, by the Lemma \ref{lem:compare F and L}\eqref{lem:F < L for positive},
	$$\sum_{n=0}^{\infty}\mathcal{F}_{H,\beta}^n(\tau_0)(a)\leq\sum_{n=0}^{\infty}(\CL_{H,\beta}^{*})^n(\tau_0)(a).$$
	Now
	$$\lim_{n\to\infty}(|(\CL_{H,\beta}^{*})^n(\tau_0)(a)|)^{1/n}\leq\lim_{n\to\infty} \|(\CL_{H,\beta}^{*})^n\|^{1/n} \|a\|^{1/n}=\rho(\beta)<1$$
	and by the root test $\sum_{n=0}^{\infty}\mathcal{F}_{H,\beta}^n(\tau_0)(a)$ converges absolutely. The fact that this convergence is absolute then implies that $\sum_{n=0}^{\infty}\mathcal{F}_{H,\beta}^n(\tau_0)$ converges in $C(K)^*$ with the weak*-topology and that the map that sends $\tau_0$ to $\sum_{n=0}^{\infty}\mathcal{F}_{H,\beta}^n(\tau_0)$ preserves convex combinations.
\end{proof}

\begin{proposition}\label{p:KMS.beta.grande}
	Let $\beta_c\in(\beta,\infty]$.
	\begin{enumerate}[(i)]
		\item\label{i:T beta grande} There is a bijection preserving extreme points between $K_\beta(\CT_\Gamma)$ and the set $S(C(K))$ of states on $C(K)$ that sends $\varphi\in K_\beta(\CT_\Gamma)$ to $(\tau(1)-\mathcal{F}_{H,\beta}(\tau)(1))^{-1}(\tau-\mathcal{F}_{H,\beta}(\tau))$, where $\tau=\varphi|_{C(K)}$. In particular, the extreme points of $K_\beta(\CO_\Gamma)$ are parametrized by $K$.
		\item\label{i:O beta grande} The above map restricts to a bijection between $K_\beta(\CO_\Gamma)$ and the set $S(C(K)/J_X)$ of states on $C(K)$ vanishing $J_X$. In particular, the extreme points of $K_\beta(\CO_\Gamma)$ are parametrized by $B\IFS$.
	\end{enumerate}
\end{proposition}

\begin{proof}
	By Lemma~\ref{lem:betagrande}\eqref{i:noinf}, we have that $K_\beta(\CT_\Gamma)_i=K_\beta(\CO_\Gamma)_i=\emptyset$, hence we only have to deal with KMS states of finite type.
	
	\eqref{i:T beta grande} Let $\varphi\in K_\beta(\CT_\Gamma)_f$ and $\tau=\varphi|_{C(K)}$. By Lemma \ref{lem:compare F and L}\eqref{lem:F < L for positive}, we have that \begin{equation}\label{eq:difpos}\CF_{H,\beta}(\tau)(1)\leq\CL_{H,\beta}^*(\tau)(1)\leq\rho(\beta)<1=\tau(1),\end{equation} so that $\tau(1)-\CF_{H,\beta}(\tau)(1)>0$. Condition (K2) of Theorem~\ref{thm:LN} then implies that $\tau_0=(\tau(1)-\mathcal{F}_{H,\beta}(\tau)(1))^{-1}(\tau-\mathcal{F}_{H,\beta}(\tau))$ is state on $C(K)$.
	
	Let now $\tau_0\in S(C(K))$. By Lemma~\ref{lem:betagrande}\eqref{i:allfin}, $\omega=\sum_{n=0}^{\infty}\mathcal{F}_{H,\beta}^n(\tau_0)\in C(K)^*$ so that $\tau=(\omega(1))^{-1}\omega\in S(C(K))$. Notice that $\tau-\CF_{H,\beta}(\tau)=(\omega(1))^{-1}\tau_0$ so that $\tau$ satisfies condition (K2) of Theorem~\ref{thm:LN} and hence it extends to an element $\varphi\in K_\beta(\CT_\Gamma)_f$.
	
	Straightforward computations show that these constructions are one inverse of the other. Let us show that the above constructions preserve extreme points.
	
	Suppose that $\varphi^1=\lambda\varphi^2+(1-\lambda)\varphi^3$ for $\lambda\in (0,1)$ and $\varphi^1,\varphi^2,\varphi^3\in K_\beta(\CT_\Gamma)_f$. Let $\tau^1$, $\tau^2$ and $\tau^3$ be, respectively, the restriction of $\varphi^1$, $\varphi^2$ and $\varphi^3$ to $C(K)$. Also, for $i=1,2,3$, let $\tau_0^i$ be constructed from $\tau_i$ as above. Define the constants $c_1$, $c_2$ and $c_3$ by $c_i=\tau^i(1)-\CF_{H,\beta}(\tau^i)(1)$, where $i=1,2,3$. For each $i=1,2,3$, because $\varphi^i$ is of finite type, by Equation \eqref{eq:difpos}, we have that $c_i>0$. Then
	\begin{align*}\tau^1_0&=\frac{\tau^1-\CF_{H,\beta}(\tau^1)}{\tau^1(1)-\CF_{H,\beta}(\tau^1)(1)}\\
	&=\frac{\lambda(\tau^2-\CF_{H,\beta}(\tau^2))+(1-\lambda)(\tau^3-\mathcal{F}_{H,\beta}(\tau^3))}{c_1}\\
	&=\frac{\lambda c_2(\tau^2-\mathcal{F}_{H,\beta}(\tau^2))}{c_1c_2}+\frac{(1-\lambda)c_3(\tau^3-\mathcal{F}_{H,\beta}(\tau^3))}{c_1c_3}\\
	&=\frac{\lambda c_2}{c_1}\tau^2_0+\frac{(1-\lambda)c_3}{c_1}\tau^3_0.
	\end{align*}
	Notice that
	$$\frac{\lambda c_2}{c_1}+\frac{(1-\lambda)c_3}{c_1}=\frac{\lambda c_2+(1-\lambda)c_3}{c_1}=\frac{c_1}{c_1}=1$$
	so that $\tau_0^1$ is a convex combination of the elements $\tau_0^2$ and $\tau_0^3$ of $S(C(K))$. If $\tau_0^1$ is an extremal point of $\mathcal{T}(C(K)/J_X)$, then $\tau_0^1=\tau_0^2=\tau_0^3$ so that $\phi^1=\phi^2=\phi^3$ is an extremal point $K_\beta(\CT_\Gamma)_f$.
	
	Similarly, using Lemma~\ref{lem:betagrande}\eqref{i:allfin}, we see that if $\varphi$ is an extreme point of $K_\beta(\CT_\Gamma)_f$, then the corresponding $\tau_0$ is an extreme point of $S(C(K))$. For the last part of the statement, it is well known that the extreme points of $S(C(K))$ are the pure states which are given by the points of $K$.
	
	\eqref{i:O beta grande} In the above construction, it is clear that $\tau$ satisfies (K1) of Theorem~\ref{thm:LN} if and only if $\tau_0$ vanishes on $J_X$. By Lemma~\ref{lemma:FinBranchIdeal}, if that is the case $\tau_0$ has support on $B\IFS$. Since we are assuming that $B\IFS$ is finite, the extreme points of $S(C(K)/J_X)$ are exactly the delta Dirac measures $\delta_y$ for $y\in B\IFS$.
\end{proof}

For $\beta\in (0,\beta_c]$, we impose an extra condition as done in \cite[Section 6]{IKW}. We start with a few lemmas.

\begin{lemma}\label{lem:point.mass.orbit}
	Let $\tau$ be a state on $C(K)$ satisfying (K2) of Theorem~\ref{thm:LN}. If $\tau$ has a point mass at $x$ then it has point mass at $y$ for all $y\in O(x)$. 
\end{lemma}

\begin{proof}
	If $\tau$ satisfies (K2) of Theorem~\ref{thm:LN} then, $\tau\geq\CF_{H,\beta}(\tau)\geq\tau(\{x\})\CF_{H,\beta}(\delta_x)$, where $\delta_x$ is the Dirac delta at $x$. For $a\in C(K)$, we have
	$$\CF_{H,\beta}(\delta_x)(a)=\int_K \CS_{H,\beta}(a)d\delta_x=\sum_{j=1}^d \frac{h_j^{-\beta}(\gamma_j(x))}{e(\gamma_j(x),x)}a(\gamma_j(x))$$
	so that
	$$\CF_{H,\beta}(\delta_x)=\sum_{j=1}^d \frac{h_j^{-\beta}(\gamma_j(x))}{e(\gamma_j(x),x)}\delta_{\gamma_j(x)}.$$
	
	It follows that
	$$\tau(\{\gamma_j(x)\})\geq\tau(\{x\})\CF_{H,\beta}(\delta_x)(\{\gamma_j(x)\})=\tau(\{x\})\frac{h_j^{-\beta}(\gamma_j(x))}{e(\gamma_j(x),x)}$$
	so that if $\tau(\{x\})>0$ then $\tau(\{\gamma_j(x)\})>0$.
\end{proof}

\begin{lemma}
	Let $x\in K$. If $O(x)\cap C(\gamma_1,\ldots,\gamma_n)=\emptyset$ then
	\begin{equation}\label{lem:F=L delta}
		\CF_{H,\beta}^n(\delta_x)=(\CL_{H,\beta}^{*})^n(\delta_x)
	\end{equation}
	for all $n\in\mathbb{N}$.
\end{lemma}

\begin{proof}
	If $O(x)\cap C(\gamma_1,\ldots,\gamma_n)=\emptyset$ then for all $y\in O(x)$ we have that $e(\gamma_i(y),y)=1$. Because of this, when we calculate both sides of \eqref{lem:F=L delta}, we obtain
	$$\sum_{i_1,\ldots,i_n=1}^d h_{i_1}^{-\beta}(\gamma_{i_1}(x))h_{i_2}^{-\beta}(\gamma_{i_1}\circ\gamma_{i_2}(x))\cdots h_{i_n}^{-\beta}(\gamma_{i_1}\circ\cdots\circ\gamma_{i_n}(x))\delta_{\gamma_{i_1}\circ\cdots\circ\gamma_{i_n}(x)},$$
	and hence $\CF_{H,\beta}^n(\delta_x)=(\CL_{H,\beta}^{*})^n(\delta_x)$.
\end{proof}

\begin{proposition}\label{p:KMS.beta.pequeno} Suppose that $\Gamma$ satisfies the escape condition and let $\beta>0$.
	\begin{enumerate}[(i)]
		\item\label{prop:beta < beta_c} If $\beta<\beta_c$ then $K_{\beta}(\CT_\Gamma)=K_{\beta}(\CO_\Gamma)=\emptyset$.
		\item\label{prop:beta = beta_c} If $\beta=\beta_c$ then there is a unique $(\sigma,\beta)$-KMS state both on $\CT_\Gamma$ and $\CO_\Gamma$, which is of infinite type and is given by the unique state $\tau\in C(K)^*$ such that $\CL_{H,\beta}^{*}(\tau)=\tau$.
	\end{enumerate}
\end{proposition}

\begin{proof}
	\eqref{prop:beta < beta_c} Recall that $\beta<\beta_c$ is equivalent to $\rho(\beta)>1$. Fix $\tau$ satisfying (K2) of Theorem~\ref{thm:LN}. Let us first show that $\tau(C(\gamma_1,\ldots,\gamma_n))=0$. As $C\IFS$ is finite, if we suppose that $\tau(C(\gamma_1,\ldots,\gamma_n))>0$ then $\tau$ would have point mass at a point $y\in C$. Using the escape condition, take $x\in O(y)$ such that $O(x)\cap C(\gamma_1,\ldots,\gamma_n)=\emptyset$. By Lemma \ref{lem:point.mass.orbit}, we have that $\tau(\{x\})>0$. Let $k:=k_\beta$ be given as in Theorem \ref{thm:RPF}, then
	\begin{align*}
	\tau(k)\geq\tau(\{x\})\CF_{H,\beta}^n(\delta_x)(k)&=\tau(\{x\})(\CL_{H,\beta}^{*})^n(\delta_x)(k)=\tau(\{x\})\delta_x(\CL_{H,\beta}^n(k))\\
	&=\tau(\{x\})\delta_x(\rho(\beta)^n k)=\rho(\beta)^n \tau(\{x\}) k(x) \xrightarrow{n\to\infty} \infty
	\end{align*}
	which is a contradiction.
	
	Now, if $\tau$ is of infinite type with respect to $(H,\beta)$ then, by Proposition \ref{prop:inf=eigen}, $\CL_{H,\beta}^{*}(\tau)=\tau$. Let us see that this gives a contradiction. Let $\tau_{\beta}$ and $k_{\beta}$ as in Theorem \ref{thm:RPF}, then by this same theorem, $\rho(\beta)^{-n}\tau=\rho(\beta)^{-n}(\CL_{H,\beta}^{*})^n(\tau)$ converges to $\tau(k_\beta)\tau_\beta$ in the weak*-topology. This implies that $\tau(k_\beta)=0$. On the other hand, because $k_\beta>0$ and $K$ is compact, there exists a real number $c>0$ such that $k_\beta\geq c$ and hence $\tau(k_\beta)\geq c>0$ arriving at contradiction.
	
	If $\tau$ is of finite type with respect to $(H,\beta)$ then $\tau=\sum_{n=0}^{\infty}\CF_{H,\beta}^n(\tau_0)$ for a finite trace $\tau_0$. In fact $\tau_0\in C(K)^*_+$ since we are assuming that $\tau$ satisfies (K2). Observe that $\tau_0(C\IFS)=0$, so by Lemma \ref{lem:compare F and L}\eqref{lem:F=L if tau(C)=0}, $\CF_{H,\beta}^n(\tau_0)=(\CL_{H,\beta}^{*})^n(\tau_0)$ for all $n$. Now, applying $\tau$ in $k_\beta$, we have
	\begin{equation}
	\begin{split} \label{eq:DivSeriesFinType}
		\tau(k_\beta)&=\sum_{n=0}^{\infty}\CF_{H,\beta}^n(\tau_0)(k_\beta)=\sum_{n=0}^{\infty}(\CL_{H,\beta}^{*})^n(\tau_0)(k_\beta)=\sum_{n=0}^{\infty}\tau_0(\CL_{H,\beta}^n(k_\beta))\\
		&= \sum_{n=0}^{\infty}\tau_0(\rho(\beta)^n k_\beta)=\tau_0(k_\beta)\sum_{n=0}^{\infty}\rho(\beta)^n=\infty 
	\end{split}
	\end{equation}
	so that we don't have a convergence in the weak*-topology, which is a contradiction.
	
	\eqref{prop:beta = beta_c} Now, we recall that $\beta=\beta_c$ is equivalent to $\rho(\beta)=1$. We first notice that Equation \eqref{eq:DivSeriesFinType} is also valid for $\rho(\beta)=1$ so that we do not have KMS states of finite type in this case. By Proposition \ref{prop:inf=eigen} and Theorem \ref{thm:LN}, if $\tau$ satisfies $\CL_{H,\beta}^{*}(\tau)=\tau$ then it extends us a KMS state, which is necessarily of infinite type, both on $\CT_\Gamma$ and $\CO_\Gamma$.
		
	Finally, we have to show that the restriction $\tau=\varphi|_{C(K)}$ of a KMS state of infinite type $\varphi$ satisfies $\CL_{H,\beta}^{*}(\tau)=\tau$. This follows from Proposition \ref{prop:inf=eigen} once we show that $\tau(C\IFS)=0$. For $k_\beta$ the eigenfunction of $\CL_{H,\beta}$, we have
	\begin{align*}
		0&=\tau(k_\beta)-\CF_{H,\beta}(\tau)(k_\beta)=\tau(k_\beta-\CS_{H,\beta}(k_\beta))=\tau(\mathcal{L}_{H,\beta}(k_\beta)-\CS_{H,\beta}(k_\beta))\\
		&=\int_K\sum_{j=1}^d\left(1-\frac{1}{e(\gamma_j(x),x)}\right)h_j^{-\beta}(\gamma_j(x))k_\beta(\gamma_j(x))d\tau(x)
	\end{align*}
	which implies that $\tau(C\IFS)=0$ because $C\IFS$ is finite and $h_jk_\beta>0$ for all $j=1,\ldots,d$. Hence $\CL_{H,\beta}^{*}(\tau)=\tau$. The uniqueness of such $\tau$ is given by Theorem \ref{thm:RPF}.
\end{proof}

We join the partial results to state the main theorem of this paper.

\begin{theorem}\label{thm:KMS}
	Let $\Gamma=\IFS$ be a contractive iterated function system with attractor $K$ satisfying the finite branches condition, where $d\geq 2$. Let $H=\{h_1,\ldots,h_d\}$ be a family of continuous functions on $K$ such that $h_j>1$ and $\log h_j$ is Dini-continuous for every $j=1,\ldots,d$. Suppose that $H$ is compatible and let $\sigma$ be the generalized action on $\CO_\Gamma$ and $\CT_\Gamma$ given by $H$. Let $\beta_c$ be the inverse critical temperature for $H$.
	\begin{enumerate}[(i)]
		\item If $\beta\in (\beta_c,\infty]$, then $K_\beta(\CO_X)_i=K_\beta(\CT_X)_i=\emptyset$. Moreover there is a one-to-one correspondence between the extreme points of $K_\beta(\CO_X)$ and the points of $B(\gamma_1,\ldots,\gamma_d)$ and one-to-one correspondence between the extreme points of $K_\beta(\CT_X)$ and the elements of $K$.
	\end{enumerate}
	Assuming also that $\Gamma$ satisfies the escape condition:
	\begin{enumerate}[(i)]\setcounter{enumi}{1}
		\item If $\beta\in(0,\beta_c)$ then $K_\beta(\CO_X)=K_\beta(\CT_X)=\emptyset$.
		\item If $\beta=\beta_c$ then there is a unique $(\sigma,\beta)$-KMS state both on $\CO_X$ and $\CT_X$, which is of infinite type and is given by the unique state $\tau\in C(K)^*$ such that $\CL_{H,\beta}^{*}(\tau)=\tau$.
	\end{enumerate}
\end{theorem}

\begin{proof}
	The theorem follows from Proposition~\ref{p:KMS.beta.grande} and \ref{p:KMS.beta.pequeno}.
\end{proof}

\begin{example}[Gauge action]
	For an iterated function system satisfying the conditions of Theorem \ref{thm:KMS}, if we define $h_j(x)=e$ for all $x\in K$ and all $j=1,\ldots,n$, then $h_j>1$ and $\log h_j$ is Dini-continuous for every $j=1,\ldots,d$. This means that Theorem~\ref{thm:KMS} is a direct generalization of \cite[Theorem 6.6]{IKW} dealing with KMS states for the gauge acion on $\CO_\Gamma$. In this case $\beta_c=\log d$. In \cite{IKW}, they do not explore KMS states on $\CT_\Gamma$. Applying Theorem~\ \ref{thm:KMS}, for $\log d$, again we a unique KMS state on $\CT_\Gamma$. And for $\beta > \log d$, the KMS states are parametrized by the points of $K$.
\end{example}

\begin{example}[Tent map]
	Let $K=[0,1]$ and consider the maps $\gamma_1(y)=1/2y$ and $\gamma_2(y)=1-1/2y$. Then $\Gamma=(\gamma_1,\gamma_2)$ is an iterated function system such that $C(\gamma_1,\gamma_2)=\{1\}$ and $B(\gamma_1,\gamma_2)=\{1/2\}$. As observed in \cite[Example 6.8]{IKW}, this system satisfies the escape condition because $O(1)\cap C(\gamma_1,\gamma_2)$. If we let $H=\{h_1,h_2\}$ such that $h_j>1$ and $\log h_j$ is Dini-continuous for every $j=1,2$, and $h_1(1/2)=h_2(1/2)$, we can apply Theorem \ref{thm:KMS}.
	
	Since $\gamma_1$ and $\gamma_2$ are the inverse branches of the tent map, $\CO_\Gamma$ can be written as an Exel's crossed product by endomorphism \cite[Theorem~3.22]{Gil1}. Because $h_1(1/2)=h_2(1/2)$, we can define a function $h\in C(K)$ by applying $h_1$ on $[0,1/2]$ and $h_2$ on $[1/2,1]$. We could then try to apply \cite[Theorem 9.6]{E2}, however, because of the branched point a key hypothesis of \cite[Theorem 9.6]{E2} is not satisfied, namely, the corresponding conditional expectation is not of finite type. In a sense, \cite[Theorem 9.6]{E2} does not detect KMS states arising from branched points.
\end{example}

\begin{example}[Graph separation condition]
	If the iterated function system $\Gamma$ satisfies de graph separation condition, then $\CO_\Gamma$ is isomorphic the Cuntz algebra $\CO_d$ \cite[Proposition 4.1]{KW1}. Usually, the study of KMS states on Cuntz algebras is related with measures on the full shift $\Sigma_d$ \cite[Section 4.2]{IK}. In our case, we case use the code map of Proposition \ref{prop:CodeMap} to see $C(K)$ as a subalgebra of $\CO_d$ \cite[Proposition~3.17]{Gil1}, so that KMS states on $\CO_d$ will be given by measures on $K$, which can be for instance the unit interval \cite[Example~4.3]{KW1} or the Sierpiński gasket \cite[Example~4.5]{KW1}. In fact, the proof of the Ruelle-Perron-Frobenius given on \cite{FL} relies on the code map and relates measures of $K$ and $\Sigma_d$.
\end{example}

\bibliographystyle{abbrv}
\bibliography{ref}

\end{document}